\documentclass[12pt, reqno]{amsart}
\usepackage{fullpage,url}
\usepackage{graphicx}

\DeclareMathOperator{\Pic}{Pic}
\DeclareMathOperator{\Proj}{Proj}
\DeclareMathOperator{\Aut}{Aut}
\DeclareMathOperator{\Gal}{Gal}
\DeclareMathOperator{\Br}{Br}
\DeclareMathOperator{\Hom}{Hom}
\DeclareMathOperator{\im}{im}
\DeclareMathOperator{\PGL}{PGL}
\DeclareMathOperator{\SL}{SL}

\newtheorem{theorem}{Theorem}
\newtheorem*{theorem*}{Theorem}

\newtheorem{observation}{Observation}

\newtheorem{proposition}[theorem]{Proposition}

\usepackage[backrefs,lite,nobysame]{amsrefs} 

\begin{document}

\title[short paper title]{On the arithmetic of one del Pezzo surface\\ over the field with three elements}

\author{Nikita Kozin}
\address{Department of Mathematics, Rice University, 
	Houston, TX, USA}
\email{nikita.kozin@rice.edu}
\author{Deepak Majeti}
\address{Department of Computer Science, Rice University, 
	Houston, TX, USA}
\email{deepak.majeti@rice.edu}


\begin{abstract}
We discuss the problem of existence of rational curves on a certain del Pezzo surface from a computational point of view and suggest a computer algorithm implementing search. In particular, our computations reveal that the surface contains 920 rational curves with parametrizations of degree $8$ and does not contain rational curves for a smaller degree.
\end{abstract}

\maketitle
\section{Introduction}\label{S:introduction}

\noindent
When one studies varieties over finite or number fields many questions about their rational points can be approached by a direct computation. These problems include computing rational points of bounded height, visualizing them or checking hypotheses on their structure (see e.g. \cite{Tschinkel} and \cite{Elsenhans_Cubic}).

The problem of finding parametrizations for rational curves on varieties is similar to the above. For the case of a plane curve given by the equation $f(x,y)=0$ the problem consists in finding a pair of rational functions in one variable $x=x(t)$ and $y=y(t)$ that would satisfy the defining equation:
\[
f(x(t),y(t))=0
\]
The reverse problem is called implicitization and in both cases looking for an effective algorithm that computes precise formulas is an active area of research \cite{RatAlgCurves}.

For higher dimensional varieties the existence of a rational curve implies the existence of a rational point.
Koll\'ar \cite{Kollar08} has analyzed the case of projective cubic hypersurfaces in $\mathbb P^n$ of degree $\le n$. For these hypersurfaces the existence of a rational point follows from the Chevalley-Warning theorem. In particular, Koll\'ar shows that a smooth cubic surface in $\mathbb P^3$ contains a rational curve of degree at most 216 through every point, given a large enough base field. Generally, rational parametrizations of curves on a variety $X$ can be considered as rational points in spaces $\Hom_d(\mathbb P^1, X)$, where $d$ is the degree of a map. These spaces are not compact and might contain no rational points. Therefore, it is natural to ask, what is the minimal value of $d$ for which the corresponding space has a rational point and how many points does it have.

In this paper we consider the problem of finding rational parametrizations of curves on a certain del Pezzo surface given by one equation over the field $\mathbb F_3$. The motivation comes from the recent result by Salgado, Testa and V{\'a}rilly-Alvarado who proved the unirationality of every smooth degree 2 del Pezzo surface over finite fields, with three exceptions. In particular, their construction derives the unirationality of the surface from the existence of a rational curve on it. Recently, Festi and van Luijk confirmed the unirationality of the remaining three cases by exhibiting a rational curve for every surface. Our interest lies in a~computational aspect of the problem. In this work we present an algorithm for finding all possible rational parametrizations for a given degree.


\section{Del Pezzo surfaces and unirationality}\label{S:motivation}
First, we review relevant notions and facts from the arithmetic geometry. A {\it{del Pezzo surface}} is a smooth projective algebraic variety of dimension two whose anticanonical class is ample. In particular, for $r\ge 3$ there exists an embedding of the surface as a degree $r$ surface in a projective space. It is known that the only possible values for the degree are $1\le r\le 9$. An $n$-dimensional variety $X$ over a field $k$ is called {\it unirational} if there exists a dominant map $\mathbb P^n \to X$ defined over $k$. This can be restated by saying that the field of functions of the surface $X$ is a subfield of a pure transcendental extension of $k$. A general reference on facts about del Pezzo surfaces is \cite{Manin}.

It is known that every del Pezzo surface of degree $r\ge 3$ containing a point is unirational (see e.g. \cite{Kollar02} for $r=3$). By a result of Manin \cite[Theorem 29.4]{Manin} a degree 2 del Pezzo surface is unirational given that it has a point not lying neither on exceptional curves nor on the ramification curve, which we discuss later. Recently, Salgado, Testa and V{\'a}rilly-Alvarado \cite{Alvarado} proved that for the case when $k$ is a finite field the answer is still positive with possibly three exceptions. In particular, following ideas of Manin, they show that the required dominant map can be constructed from a non-constant morphism
\begin{equation} \label{eq:curve}
f: \mathbb P^1 \longrightarrow X,
\end{equation}
and then make further analysis including the direct check of cases when such a map exists.

The image $f(\mathbb P^1)$ defines a {\it rational curve} on a surface $X$, by which we mean 1-dimensional geometrically integral subvariety of geometric genus 0. In particular, finding a map (\ref{eq:curve}) would imply unirationality of the surface. One of the exceptional three cases is the surface $X$ defined by the equation
\begin{equation}
-w^2 = x^4 + y^3z-yz^3 \label{eq:2}\\
\end{equation}
over the field $k=\mathbb F_3$ with three elements in a {\it weighted projective space} (see \cite{Harris} for a definition):
\[
X \hookrightarrow \mathbb P_k(1,1,1,2) = \Proj(k[x,y,z,w]).
\]
Therefore, every morphism (\ref{eq:curve}) can be considered as the map to the weighted projective space and it can be shown \cite[Exercise 1.3.10]{Kollar} that every morphism to this space is given by homogeneous polynomials in two variables
\begin{equation} \label{eq:3}
x = x(s,t), \quad y=y(s,t), \quad z=z(s,t), \quad w=w(s,t),
\end{equation}
where $x,y,z$ are of some degree $d$, and $w$ is of degree $2d$. That the image of the morphism lies in a surface $X$ means that expressions from (\ref{eq:3}) being substituted into equation (\ref{eq:2}) must satisfy it. If such a morphism (and hence a rational curve on $X$) exists, then by setting $s=1$ we obtain one-variable parametrizations
\begin{equation} \label{eq:4}
x = x(t), \quad y=y(t), \quad z=z(t), \quad w=w(t),
\end{equation}
such that $x,y,z$ are polynomials of degree $\le d$ and $w$ is a polynomial of degree $\le 2d$.

Since $k=\mathbb F_3$ is a finite field, there are only finitely many polynomials with coefficients in $k$ of a given degree $d$ and finding parametrizations (\ref{eq:4}) can be attacked by brute-force on a computer, at least for small values of $d$. This fact motivates us to formulate the question about unirationality of the surface $X$ from a computational point of view in the following way:
\vspace{1mm}
\begin{center}
\textit{For a given $d$, determine all possible polynomial parametrizations (\ref{eq:4}) that satisfy (\ref{eq:2}).}
\end{center}
\vspace{1mm}
The weighted projective space $\mathbb P(1,1,1,2)$ embeds into the usual projective space $\mathbb P^6$ as a complete intersection of three quadrics and in turn gives a projective embedding of the surface $X$ corresponding to the linear system $|{-2}K|$. If we have a rational curve parametrized by degree $(d,d,d,2d)$ homogeneous polynomials then the composite map
\[
\mathbb P^1 \rightarrow X \subset \mathbb P(1,1,1,2) \hookrightarrow \mathbb P^6
\]
maps $\mathbb P^1$ onto a degree $2d$ projective curve. In this paper by the {\it degree} of a rational curve we mean the value $d$ of the corresponding parametrization. The Picard group of the surface $X$ is isomorphic to $\mathbb Z$ and is generated by the anticanonical class. Any rational curve of degree $d$ would correspond to class $-(d/2)K$. In particular, looking for curves it makes sense to check among parametrizations of even degrees.

Precomposing the morphism $f$ with any automorphism of the curve would give different parametrization which, nevertheless, would correspond to the same curve on a surface. In particular, since the group of automorphisms $\Aut(\mathbb P^1)$ over $\mathbb F_3$ is isomorphic to $\mathfrak S_4$, the symmetric group on four letters, we conclude that to every curve on $X$ there would correspond 24 different parametrizations.

The number of polynomials grows exponentially with $d$, for example over the field $\mathbb F_3$, $d=8$ would require to check
\[
3^9\times 3^9 \times 3^9 \times 3^{17} \quad = \quad 3^{34} \quad \sim \quad 1.7 \times 10^{16}
\]
possible quadruples of polynomials. Thus the naive brute-force approach is not feasible.  We use some of the particular arithmetic properties of $X$ to reduce and optimize the brute-force approach (see Section 5 for more details).  In Section 6, we provide an algorithm based on this approach; the algorithm demonstrates that there are no rational curves for $d \le 7$ while for $d = 8$ there are exactly 920 (up to projective automorphisms) rational curves of degree 2d on $X$.

\section{Structure of the geometric Picard module}
The geometric Picard group $\Pic(\overline X)$ is isomorphic to $\mathbb Z^8$, as after passing to the algebraic closure of the ground field, $\overline X$ becomes isomorphic to the blow-up of $\mathbb P^2$ at seven points. In particular, $\Pic(\overline X)$ is freely generated by the classes of exceptional divisors and the pull-back of the class of the line on $\mathbb P^2$. On the other hand, $\overline X$ contains 56 exceptional curves (which become conics under the $-2K$-embedding) that also generate the Picard group. Therefore, to understand the action of Galois group on $\Pic(\overline X)$ it is sufficient to analyze its action on the 56 exceptional curves.

The linear system of the anticanonical class has dimension 3 and gives a map onto $\mathbb P^2$ which is a double cover branched along the smooth quartic. The equation of the quartic is the right part of (\ref{eq:2}):
\begin{equation} \label{eq:quartic}
x^4-yz^3+y^3z = 0.
\end{equation}
All 56 exceptional curves on the surface are the pre-images of the 28 {\it bitangent lines} to the branch quartic, that is lines having even intersection multiplicity with quartic at every common point. In particular, that means that extra relation induced by the line equation
\[
Ax+By+Cz=0
\]
would factor (\ref{eq:quartic}) up to constant into the product of two squares.
One case to look for bitangent lines is when the coefficient of $x$ is zero, $A=0$.  Here, we readily obtain four lines given by
\[
y = z \qquad y=-z \qquad y=0 \qquad z=0
\]
whose pre-images on the surface are given by
\[
L_{y=z}^\pm:
\begin{cases}
y=z\\
w = \pm \sqrt{2} x^2
\end{cases}
L_{y=-z}^\pm:
\begin{cases}
y=-z\\
w = \pm \sqrt{2} x^2
\end{cases}
L_{y=0}^\pm:
\begin{cases}
y=0\\
w = \pm \sqrt{2} x^2
\end{cases}
L_{y=0}^\pm:
\begin{cases}
z=0\\
w = \pm \sqrt{2} x^2
\end{cases}
\]

\vspace{-3mm}
\begin{figure}[hb]
\centerline{
\includegraphics[scale=0.75]{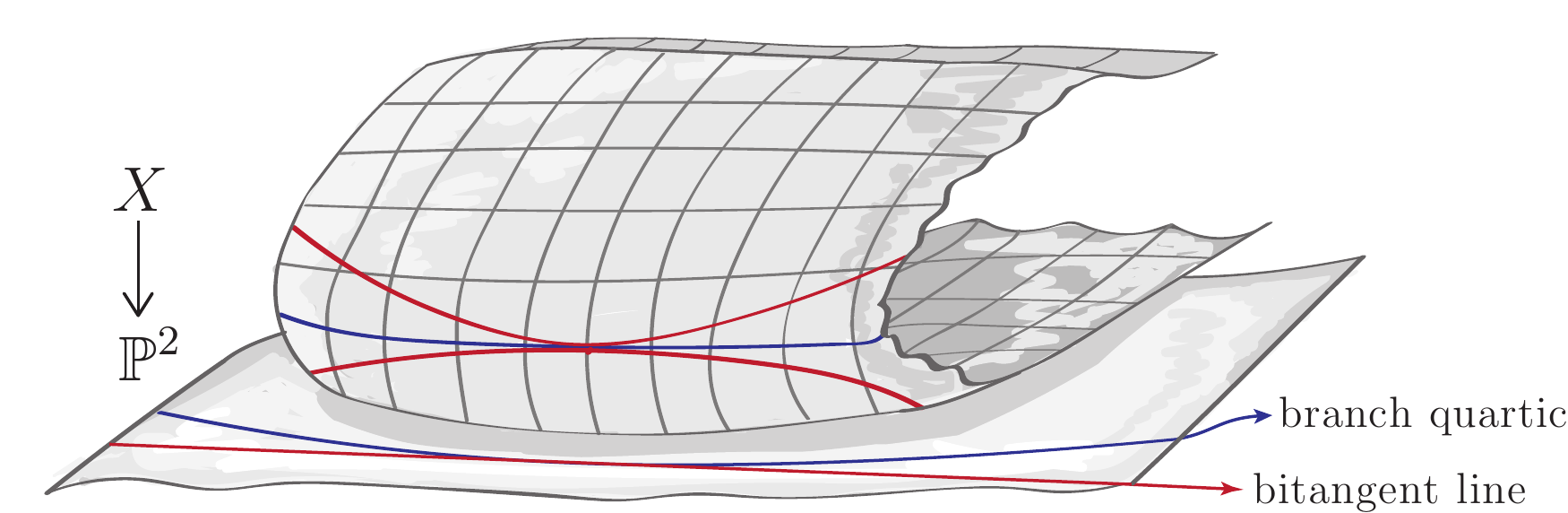}
}
\caption {Double cover of projective plane}
\end{figure}

The remaining $24$ bitangent lines must correspond to a non-zero coefficient $A\ne 0$. We will look for them in the form
\[
x = az + by.
\]
Plugging those into the equation of quartic and setting $y=1$ gives
\[
(az + b)^4 - z^3 + z=0
\]
which can be rewritten as
\begin{equation} \label{eq:expansion1}
a^4\left[z^4 + \left( \frac{b}{a} - \frac{1}{a^4}\right) z^3 + \left( \frac{b^3}{a^3} + \frac{1}{a^4}\right)z + \frac{b^4}{a^4}\right] = 0.
\end{equation}
Now, if it were to factor into squares
\[
a^4(z+\alpha)^2(z+\beta)^2 = 0,
\]
we would have
\[
a^4(z^4+(2\alpha + 2\beta)z^3 + (\alpha^2 + \alpha\beta + \beta^2)z^2 + (2\alpha\beta^2 + 2\alpha^2\beta)z + \alpha^2\beta^2) = 0.
\]
We notice that the coefficient by $z^2$ is the full square:
\[
\alpha^2 + \alpha\beta + \beta^2 = (\alpha-\beta)^2.
\]
Since this coefficient vanishes in expansion (\ref{eq:expansion1}) we deduce that $\alpha = \beta$ and look for expansion of the form
\[
a(z+a)^4 = a^4(z^4+\alpha z^3+\alpha^3z+\alpha^4) = 0.
\]
In particular, this would imply that the bitangent line intersects branch quartic exactly at one point with multiplicity 4 (see Figure 1). Comparing coefficients in the last expression with ones in (\ref{eq:expansion1}) we obtain
\begin{equation} \label{eq:ab_relations}
a^8 = -1, \qquad ab^3 - ba^3+1=0.
\end{equation}
Therefore, for a fixed $a=\sqrt[8]{-1}$, there are only three choices $b_1$, $b_2$, $b_3$ that define a bitangent line. In this case the equation (\ref{eq:quartic}) reads as
\[
-w^2 = a^4 \left(z + \left(b/a - 1/a^4\right)y\right)^4 = (x+a^5y)^4.
\]
This data defines the remaining 24 bitangent lines.

Over $\mathbb F_3$ the polynomial $x^8+1$ factors as follows:
\[
x^8 + 1 = (x^4+x^2+2)(x^4+2x^2+2).
\]
Denoting some fixed root for the first irreducible factor as $\zeta$ we notice that in the extension $\mathbb F_3(\zeta)$ the square root
\[
\sqrt{-1} = \pm (\zeta^2 + 2).
\]
All four roots of $x^4+x^2+2$ and four roots of $x^4+2x^2+2$ lie in $\mathbb F_3(\zeta)$ and are permuted cyclicly by the Frobenius automorphism that generates $\Gal( \mathbb F_3(\zeta) | \mathbb F_3) \cong \mathbb Z/4\mathbb Z$. For every choice of $a$ all three corresponding values of $b$ satisfying (\ref{eq:ab_relations}) also lie in $\mathbb F_3(\zeta)$, e.g. for $a=\zeta$, corresponding $b$'s are
\[
2\zeta^3, \quad 2\zeta^3+\zeta, \quad 2\zeta^3 + 2\zeta
\]
Therefore, all 56 exceptional curves are defined over $\mathbb F_3(\zeta)$ and we have their precise description which we summarize in the following proposition (the notation for the curves is chosen to correspond to the one in \cite{Tschinkel04}).
\begin{proposition}
The degree 2 del Pezzo surface (\ref{eq:2}) becomes rational over the degree 4 extension $\mathbb F_3(\zeta)$. The 56 exceptional curves are
\[
L_{y=z}^\pm \quad L_{y=-z}^\pm \quad L_{y=0}^\pm \quad L_{z=0}^\pm,
\]
as defined above, and
\[
L_{a,b}^\pm:
\begin{cases}
x=az+by,\\
w = \pm (\zeta^2+2) (x+a^5y)^2,
\end{cases}
\]
where $a^8=-1$ and $b$ satisfies $ab^3-a^3b+1=0$.
\end{proposition}
For concreteness, we outline all parameters $a$'s, their corresponding $b$'s and $\pm$ into the diagram shown on Figure 2. Inner boxes contain values of $a$, while 8 outer boxes contain corresponding values of $b$'s, three for each $a$. The Frobenius automorphism permutes values as shown by arrows. The values of $b$'s in boxes are ordered in such a way that $F$ permutes boxes preserving the order. Because $F(\sqrt{-1}) = -\sqrt{-1}$, Frobenius maps every curve $L^\pm$ to $L^\mp$. Every conjugate pair (that is, the pair of exceptional curves that are pre-images of the same bitangent) intersects at one point with multiplicity 2. At the same time every non-conjugate pair has either intersection 0 or 1.

\begin{figure}
\centering
\includegraphics[scale=1.2]{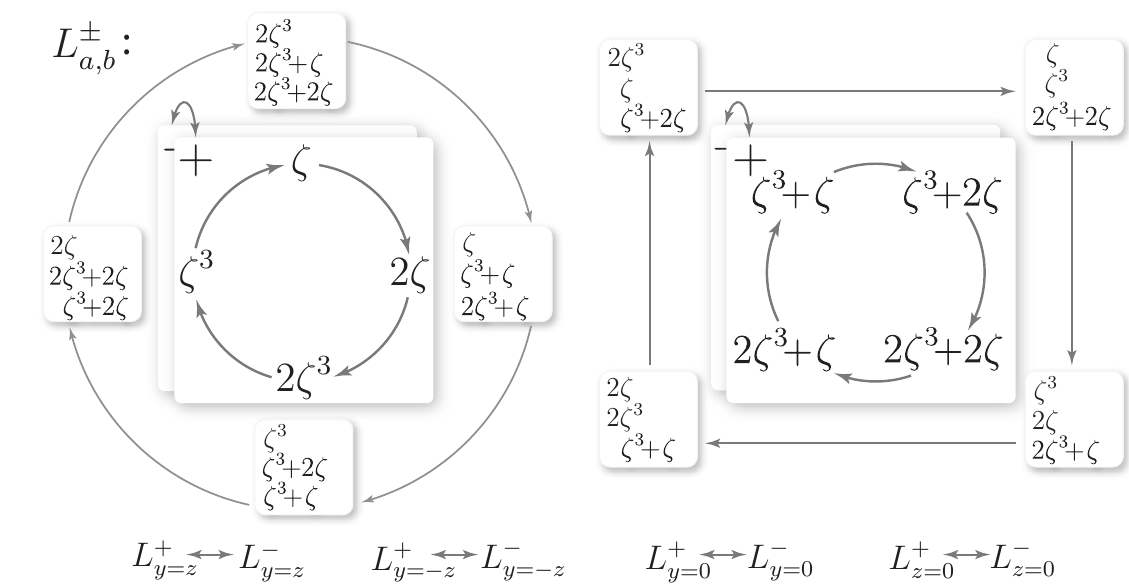}
\caption {Frobenius action on the exceptional curves}
\end{figure}

In the group of eight curves $L_{y=z}^\pm$, $L_{y=-z}^\pm$, $L_{y=0}^\pm$, $L_{z=0}^\pm$ the non-conjugate pairs intersect at one point if and only if corresponding signs $\pm$ coincide. In particular, $L_{y=z}^+$, $L_{y=-z}^+$, $L_{y=0}^+$, $L_{z=0}^+$ intersect at the point $[1:0:0:+\sqrt{-1}]$, and $L_{y=z}^-$, $L_{y=-z}^-$, $L_{y=0}^-$, $L_{z=0}^-$ intersect at the point $[1:0:0:-\sqrt{-1}]$; these points are called generalized Eckardt points. Every curve $L^\pm_{a,b}$ intersects each of $L^\pm_{y=z}$ and $L^\pm_{y=-z}$ and $L^\pm_{y=0}$ with the same $\pm$-sign at one point.

The rest of the intersections is more subtle. Direct check indicates that when $a_1\ne a_2$ every curve $L^\pm_{a_1,b_1}$ from the group of six curves $L^\pm_{a_1,b}$ intersects exactly three curves from the group of six curves $L^\pm_{a_2,b}$, one with the same $\pm$-sign and two other with the opposite $\pm$-sign. When $a_1=a_2$, the curve $L^\pm_{a_1,b_1}$ intersects exactly three curves from its group: itself and two other with the same $\pm$-sign (but different $b$'s). Finally, $L^\pm_{z=0}$ intersects exactly three curves in the group of six $L^\pm_{a,b}$.

There is no general pattern determining which $L^\pm_{a_2,b_2}$ would intersect concrete $L^\pm_{a_1,b_1}$ but for the specific pair of curves intersection point (if exists) can readily be checked by a direct computation. In particular, now we can pick concrete curves that would form an orthogonal basis for geometric Picard group.

\begin{proposition}
The geometric Picard group $\Pic(\overline X)\cong \mathbb Z^8$ is generated by classes of the following curves
\[
d_1 = L^+_{1,1} \quad d_2 = L^+_{5,2} \quad d_3 = L^+_{2,2} \quad d_4 = L^+_{4,2} \quad d_5=L^+_{6,1} \quad d_6= L^+_{8,3} \quad d_7 = L^-_{z=0}
\]
\[
d_8 = L^+_{1,1} + L^+_{5,2} + L^+_{3,3}
\]
such that the intersection numbers are $(d_i,d_i)=-1$ for $1\le i \le 7$, $(d_8,d_8)=1$ and $(d_i,d_j)=0$ for $i\ne j$. The anticanonical class is given by
\[
-K_{\overline X} = -d_1 - d_2 - d_3 - d_4 - d_5 - d_6 - d_7 + 3d_8.
\]
\end{proposition}

\noindent
Observing the action of the Frobenius automorphism on the equations of curves and computing intersections with curves generating classes we deduce that
\[
\begin{matrix}
d_1 = L^+_{1,1} & \longrightarrow & L^-_{3,1} = &-d_1-d_2 -d_4-d_6-d_7+2d_8 \\
d_2 = L^+_{5,2} & \longrightarrow & L^-_{7,2} = &-d_2-d_4 -d_5-d_6-d_7+2d_8 \\
d_3 = L^+_{2,2} & \longrightarrow & L^-_{4,2} = &-d_1-d_2 -d_3-2d_4-d_5-d_6-d_7+3d_8 \\
d_4 = L^+_{4,2} & \longrightarrow & L^-_{6,2} = &-d_4-d_7 -d_8\\
d_5 = L^+_{6,1} & \longrightarrow & L^-_{8,1} = &-d_1-d_2 -d_4-d_5-d_7+2d_8 \\
d_6 = L^+_{8,3} & \longrightarrow & L^-_{2,3} = &-d_1-d_4 -d_5-d_6-d_7+2d_8 \\
d_7 = L^-_{z=0} & \longrightarrow & L^+_{z=0} = &-d_1-d_2-d_3-d_4 -d_5-d_6-2d_7+3d_8 \\
\end{matrix}
\]
\[
\begin{matrix}
d_8 = L^+_{1,1}+L^+_{5,2}+L^+_{3,3} \longrightarrow
 L^-_{3,1}+L^-_{7,2}+L^-_{5,3} = -2d_1 -2d_2-d_3-3d_4-2d_5-2d_6-3d_7+6d_8
\end{matrix}
\]
This observation gives the action of the Galois group on $\Pic(\overline X)$.
\begin{proposition}
The action of Galois group factors through the cyclic group $\mathbb Z/4\mathbb Z$ generated by the Frobenius automorphism. The action of the generator with respect to the basis above is given by the following unimodular matrix
\begin{equation} \label{eq:mymatrix}
\begin{bmatrix}
-1& &-1& &-1&-1&-1&-2\\
-1&-1&-1& &-1& &-1&-2\\
 & &-1& & & &-1&-1 \\
-1 & -1 & -2 & -1 & -1 & -1& -1& -3 \\
 &-1&-1& &-1&-1&-1&-2 \\
-1&-1&-1& & &-1&-1&-2 \\
-1&-1&-1&-1&-1&-1&-2&-3\\
+2&+2&+3&+1&+2&+2&+3&+6
\end{bmatrix}
\end{equation}
\end{proposition}
\noindent
The trace of the above matrix equals $-2$. In our case, Weil's formula on number of points on varieties over the finite field $\mathbb F_q$:
\[
\#X(\mathbb F_q) = q^2 + q\cdot\text{trF} + 1,
\]
reads
\[
\#X(\mathbb F_3) = 3^2 +3\cdot (-2) + 1 =4
\]
and it can be easily checked that the surface contains exactly four rational points: $\thickmuskip=1mu [0:1:0:0]$, $\thickmuskip=1mu [0:0:1:0]$, $\thickmuskip=1mu [0:1:1:0]$ and $\thickmuskip=1mu [0:1:2:0]$.

\section{Brauer group and automorphisms}
The Brauer group $\Br X$ of the surface $X$ by definition is the second \'etale cohomology group $\mathrm H^2(X,\mathbb G_m)$. It can be shown that in our case this group is isomorphic to the Galois cohomology group $\mathrm{H}^1(\mathbb F_3, \Pic(\overline X))$ (see e.g. chapter 43 in \cite{Manin}). Although it is not used in the computer search for rational curves, we would like to include its computation as well as the  group of automorphisms for the sake of completeness

The result of the previous section describing the module $\Pic (\overline X))$ structure under the action of the group $\mathbb Z/4\mathbb Z$ allows us to compute it directly from projective resolution of the group $\mathbb Z$. Denoting $G=\mathbb Z/4\mathbb Z$ with fixed generator $m$, the matrix (\ref{eq:mymatrix}), we have the following resolution of $\mathbb Z$

\centerline{
\includegraphics{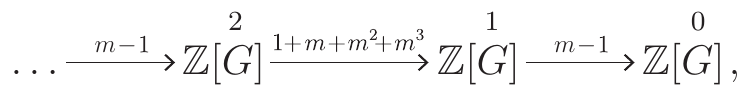}
}
\noindent
where the maps are as indicated. Since the functor of invariants $M^G$ of the modulo $M=\Pic(\overline X)=\mathbb Z^8$ is given by $\Hom_G(\mathbb Z,M)$ we can apply $\Hom_G(\cdot,M)$ to the above resolution obtaining the sequence
\[
\ldots \leftarrow \mathbb Z^8 \leftarrow \mathbb Z^8 \leftarrow \mathbb Z^8,
\]
and thus
\[
\mathrm{H}^1(G,M) \cong \ker(1+m+m^2+m^3) / \im(m-1)
\]
Direct computation indicates that $ \ker(1+m+m^2+m^3)$ and $\im(m-1)$ are generated by the columns of the matrices
\[
\begin{bmatrix}
1&&&&&&2&-1\\
&1&&&&&2&-1\\
&&1&&&&2&-1\\
&&&1&&&2&-1\\
&&&&1&&2&-1\\
&&&&&1&2&-1\\
&&&&&&3&-1
\end{bmatrix}
\qquad
\text{and}
\qquad
\begin{bmatrix}
1&&&&&3&5&-3\\
&1&&&1&2&2&-2\\
&&1&&1&3&1&-2\\
&&&1&1&3&4&-3\\
&&&&2&2&5&-3\\
&&&&&4&2&-2\\
&&&&&&6&-2
\end{bmatrix},
\]
respectively. Expressing the columns of the second matrix in terms of the columns of the first matrix allows us to compute that the quotient is isomorphic to $\mathbb Z/4\mathbb Z\times \mathbb Z/4\mathbb Z$.

Every automorphism of $X$ is either involution $w \rightarrow -w$ corresponding to the double cover induced by anticanonical class, or can be lifted to the projective automorphism of the branch quartic (\ref{eq:quartic}). The last case includes all projective plane transformations $\PGL(3,3)$ that fix the equation of quartic up to a constant multiple. As the last group has finitely many elements, these elements can be found directly. In particular, we observe that all transformation matrices with determinant 1 preserving the equation have block form
\[
\begin{bmatrix}
1 & & \\
 & a_{22} & a_{23} \\
 & a_{32} & a_{33}
\end{bmatrix}
\]
where the lower $2\times 2$ matrices are elements of the group $\SL(2,3)$ and correspond to linear transformation on variables $y$ and $z$. These transformations commute with involution. We summarize above observations in the following proposition.
\begin{proposition}
The Brauer group of the surface $X$ is isomorphic to the group $\mathbb Z/4\mathbb Z \times \mathbb Z/4\mathbb Z$, while the $\mathbb F_3$-automorphism group is isomorphic to $\mathbb Z/2\mathbb Z\times \SL(2,3)$.
\end{proposition}

\section{Arithmetic properties of the surface}
Equation (\ref{eq:2}) can be written as
\begin{equation} \label{eq:main}
x^4 + w^2 = yz^3 - y^3z
\end{equation}
such that it has two sides, left and right, each containing only two variables. A similar equation was analyzed by Elsenhans and Jahnel \cite{Elsenhans} who were interested in integer solutions to
\[
x^4 + 2y^4 = z^4 + 4w^4
\]
and attacked the problem using computer by calculating separately lists of left and right hand sides, splitting them modulo large prime $p=30011$ (so that lists can be partitioned into subsets that fit into memory), and computing their intersections. In our context solutions belong to the ring of polynomials $\mathbb F_3[t]$ rather than integers $\mathbb Z$, but the idea of pre-computing left and right hand sides and finding their set-theoretical intersection still can be applied. The search range can also be significantly reduced by noticing some arithmetical properties of the surface given by equation (\ref{eq:main}), which we state in the form of observations. Because of the isomorphism $\mathbb F_3 \cong \mathbb Z/3\mathbb Z$ it is convenient to treat polynomials as having integer coefficients modulo 3.

Notice first that the right hand side of (\ref{eq:main}) is antisymmetric:
\begin{observation}
If $(x, y, z, w)$ satisfies (\ref{eq:main}), then $(x, z, -y, w)$ also satisfies it.
\end{observation}

\noindent
Secondly, we notice that for coefficients modulo 3, the evaluated polynomial in the right hand side will always have vanishing constant term:
\begin{observation}
For any $y(t), z(t)$ the right part of (\ref{eq:main}) will have vanishing constant term.
\end{observation}
\begin{proof}
If
\[
y = a_mt^m + \ldots + a_1t + a_0, \quad z = b_nt^n + \ldots + b_1t + b_0
\]
then
\[
yz^3 - y^3z = ( \ldots + a_1t + a_0)(+ \ldots + b_1t + b_0)^3 - (\ldots + a_1t + a_0)^3(\ldots + b_1t + b_0)=
\]
\[
=\ldots + a_0b_0^3 - a_0^3b_0
\]
But cubes modulo 3 are equal to first powers by Fermat's little theorem. Hence the constant term vanishes.
\end{proof}

\noindent
On the other hand, the same argument shows that the leading term of the right hand side vanishes as well:
\begin{observation}
For any $y(t)$, $z(t)$ of degrees exactly $d$ the coefficient by $t^{4d}$ in the right hand side of (\ref{eq:main}) vanishes.
\end{observation}

\noindent
These observations allow us to purge significantly the amount of possible left hand sides: we need to keep track only of those values $x,w$ that produce vanishing constant term in the left hand side. Also, from Observation 3, it follows that, for given $d$, the values of $x$ and $w$ that produce non-zero coefficient of $t^{4d}$ can also be purged from the analysis. Moreover, Observation 2 in turn allows us to reduce the check range for the right hand sides even further. In particular, we have the following
\begin{observation}
For given $x(t), w(t)$, if the value of $x^4 +w^2$ has vanishing constant term, then the linear term also vanishes.
\end{observation}
\begin{proof}
The proof is a direct check. Given
\[
x = a_mt^m + \ldots + a_1t + a_0, \quad w = b_nt^n + \ldots + b_1t + b_0
\]
we have that modulo 3 the left hand side is
\[
x^4 + w^2 = ( \ldots + a_1t + a_0)^4 + (\ldots + b_1t + b_0)^2=
\]
\[
= \ldots (a_1a_0^3 + 2 b_0b_1)t + (a_0^4 + b_0^2)
\]
Now since $-1$ is not a square in $\mathbb F_3$ it follows that
\[
a_0^4 + b_0^2 = 0
\]
if and only if $a_0=b_0=0$ and hence the coefficient of the linear term also vanishes.
\end{proof}

\noindent
Finally, because of Observation 3, for a given $d$ we can exclude from the analysis left hand sides that have $x(t)$ or $w(t)$ of highest possible degrees:
\begin{observation}
If $x(t)$ has degree exactly $d$, or $w(t)$ has degree exactly $2d$, then the coefficient of $t^{4d}$ in the left hand side is non-zero.
\end{observation}
\begin{proof}
\[
x^4 + w^2  = (a_dt^d + \ldots)^4 + (b_{2d}t^{2d} + \ldots)^2 = (a_d^4+b_{2d}^2)t^{4d} + \ldots
\]
and because modulo 3 all non-zero squares are equal to 1, the leading coefficient is non-zero given that $a_d\ne 0$ and $b_{2d}\ne 0$.
\end{proof}

The above observations allow us to reduce significantly the lists of left and right hand sides by purging apriori incoherent data. The following section discusses our final algorithm.

\section{Algorithm and results}
The algorithm takes as the input argument the degree $d$ and consists of three steps:
\begin{enumerate}
	\item generate two sets of all admissible left and right hand sides,
	\item find set-theoretical intersection of two sets,
	\item reconstruct rational parametrizations.
\end{enumerate}
In our implementation each polynomial was precomputed from the triadic expansion of integer number and was stored in memory in terms of its coefficients, 2 bits per coefficient. This compact form highly decreased the memory requirements and allowed to do faster computations with lower memory footprint. The following picture shows the correspondence between polynomials $y,z$ and their binary representation in memory.

\vspace{2mm}
\centerline{
\includegraphics{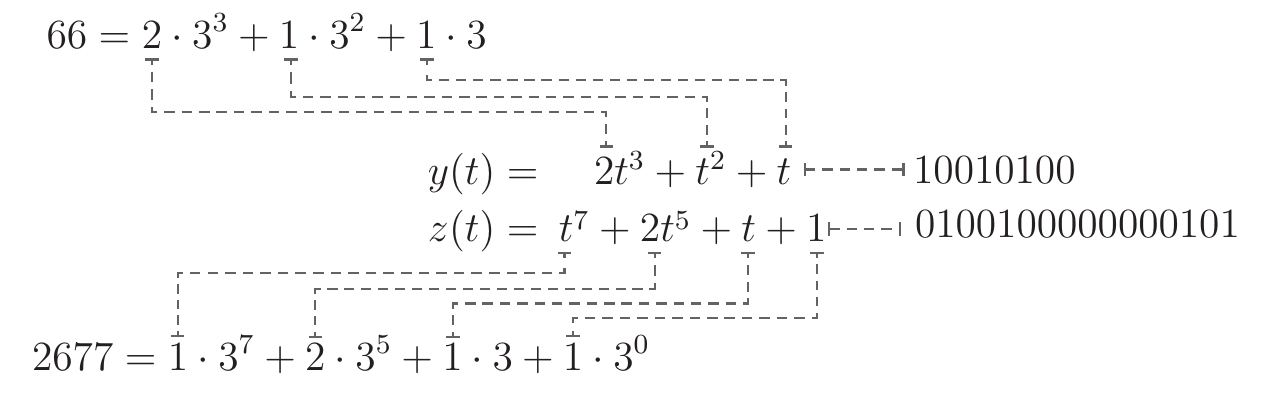}
}

At the preprocessing stage we computed lists of admissible polynomials $x,y,z,w$, then their exponents $x^4,w^2,y^3,z^3$ appearing in formula (\ref{eq:main}), and finally the list of right hand sides $z^3y-y^3z$. This process required reconstruction of polynomials from the binary form but due to the chosen encoding this conversion was efficiently performed using bit-shift operators. This list was sorted and took about 10 GB of operative memory.

The complete list of left hand sides $x^4+w^2$ did not fit into operative memory of the computer used in the experiment. We resolved this issue by generating the list dynamically: as the block of left hand sides was checked for matching, it was erased from the memory and a new block was generated. Moreover, to increase the processing speed we used parallel computing. If a match was found, the corresponding rational parametrization was reconstructed. The Figure 3 shows the scheme of our final implementation.

\begin{figure}[ht]
\centerline{
\includegraphics{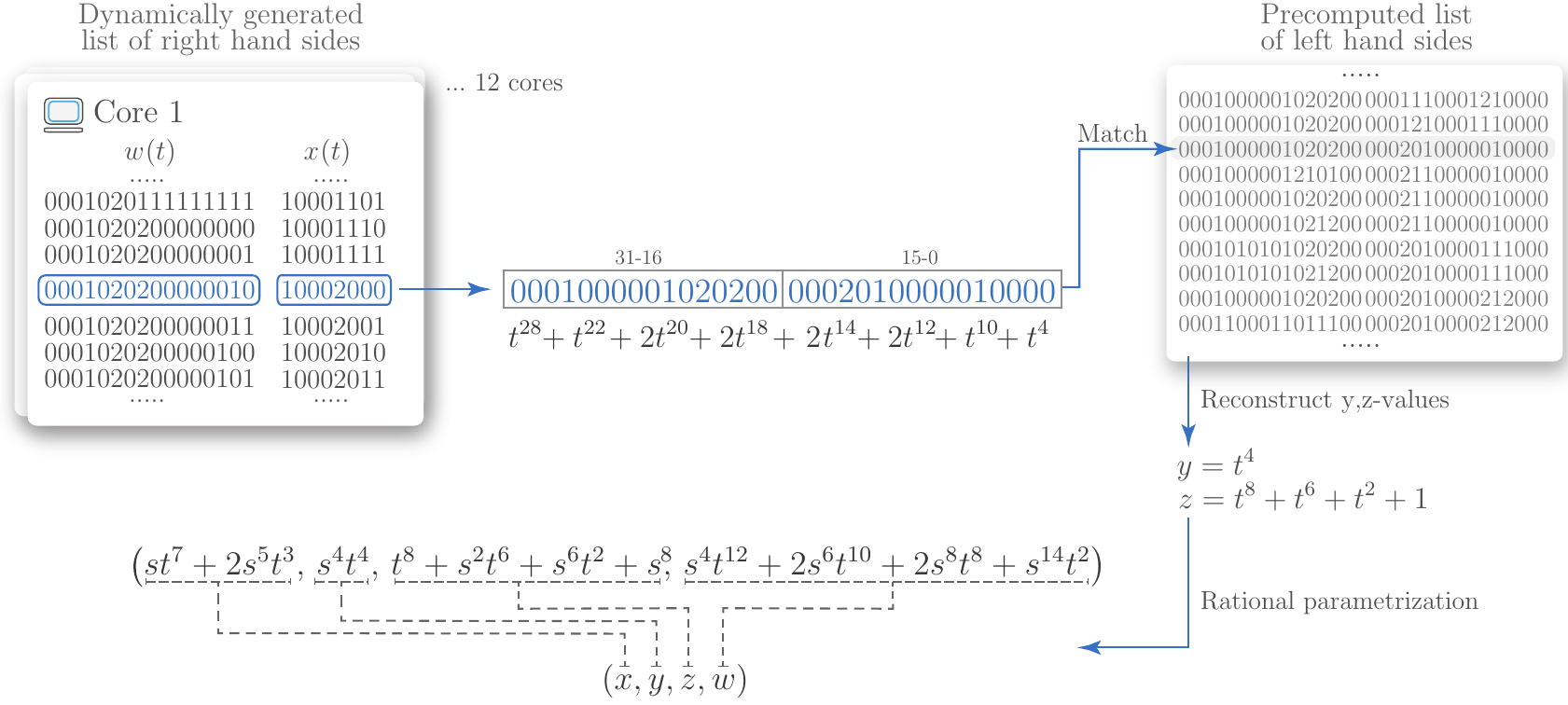}
}
\caption {The scheme of implementation}
\end{figure}

The experiment for $d\le 8$ has been conducted on Intel(R) Xeon(R) CPU with 2 processors, each with 6 cores (total of $2\times 6 = 12$ cores) running at 2.9GHz. The system had a total memory of 128 GB. Preprocessing stage took about 16 hours of CPU time while the main parallel cycle took about 14 days of CPU time. For degree $d=10$ we might expect at most $\times 27$ increase in a computational time.

The experiment revealed exactly $22080$ rational parametrizations, or equivalently $920$ rational curves, all of which had degree $d=8$. The redundancy of obtaining all parametrizations for every curve seems unavoidable. To eliminate it at the computational stage would require one to precompute the action of the group of projective automorphisms. This seems to be much a more computationally expensive task than simply removing parametrizations within the same orbit as the list is computed. The full list of curves as well as source codes are included in the arXiv posting.

The degree 8 curve discovered in \cite{vanLuijk} is also contained in our list. We want to stress the difference between our approaches. In \cite{vanLuijk} the authors first found the equation of the curve and then reconstructed the rational parametrization. It should also be noted that the key to the effective computation was partially due to our choice of surface from the three in \cite{Alvarado}. In particular, the defining equation for one of the remaining two surfaces contains coefficients from the field $\mathbb F_9$. This would make a congruence argument more difficult to implement. From the other side, we expect that an approach similar to ours can still be applied to varieties over prime fields whose defining equations contain enough symmetry.

\section*{Acknowledgments}
We thank Brendan Hassett for mentioning the problem as well as suggesting improvements. We also thank Anthony V{\'a}rilly-Alvarado and Damiano Testa for helpful conversations. Finally, we thank the referee for a thorough review and helpful remarks. The first author was supported by NSF grant 0968349.


\end{document}